\documentclass[12pt, reqno]{amsart}
\usepackage{amsmath}
\usepackage{amssymb}
\usepackage{epsfig}
\usepackage{graphicx}
\usepackage{color}
\usepackage{fullpage}
\usepackage{comment}
\definecolor{shadecolor}{gray}{0.875}
\usepackage{amscd}
\usepackage{stmaryrd}
\usepackage{threeparttable}

\usepackage{ulem}
\usepackage{amsthm}
\usepackage{csquotes}
\usepackage{enumitem}
\usepackage{graphics}
\usepackage{tikz-cd}
\usepackage{hhline}
\usepackage{multirow}
\usepackage{makecell}
\usepackage{booktabs}
\usetikzlibrary{backgrounds,calc,positioning}
\usepackage{multicol} 

\usepackage[backend=bibtex,style=alphabetic,maxbibnames=99,url=false]{biblatex}
\usepackage[colorlinks=false,urlbordercolor=white]{hyperref}

\tikzset{sgplattice/.style={inner sep=1pt,norm/.style={red!50!blue},char/.style={blue!50!black},
		lin/.style={black!50}},cnj/.style={black!50,yshift=-2.5pt,left=-1pt of #1,scale=0.5,fill=white}}
	
	\usepackage{here}
	\usepackage{caption}
	\usepackage{subcaption}

\let\oldtocsection=\tocsection

\let\oldtocsubsection=\tocsubsection

\let\oldtocsubsubsection=\tocsubsubsection

\renewcommand{\tocsection}[2]{\hspace{0em}\oldtocsection{#1}{#2}}
\renewcommand{\tocsubsection}[2]{\hspace{1em}\oldtocsubsection{#1}{#2}}
\renewcommand{\tocsubsubsection}[2]{\hspace{2em}\oldtocsubsubsection{#1}{#2}}

\makeatletter

\newcommand{\Rmnum}[1]{\expandafter\@slowromancap\romannumeral #1@}
\makeatother

\numberwithin{equation}{section}

\input xy
\xyoption{all}

\calclayout
\allowdisplaybreaks[3]

\theoremstyle{plain}
\newtheorem{prop}{Proposition}[section]

\newtheorem{theo}[prop]{Theorem}
\newtheorem{coro}[prop]{Corollary}

\newtheorem{lemm}[prop]{Lemma}

\theoremstyle{definition}
\newtheorem{defi}[prop]{Definition}

\newtheorem{rema}[prop]{Remark}

\newtheorem{mainthm}{Theorem}

\newlist{steps}{enumerate}{1}
\setlist[steps, 1]{label = Step \arabic*:}

\def\ra{\rightarrow}

\def\cC{{\mathcal C}}

\def\cG{{\mathcal G}}
\def\cH{{\mathcal H}}

\def\cO{{\mathcal O}}
\def\cP{{\mathcal P}}

\def\cV{{\mathcal V}}

\def\cX{{\mathcal X}}

\def\rk{{\mathrm{rk}}}

\def\bE{{\mathbb E}}
\def\bF{{\mathbb F}}
\def\bG{{\mathbb G}}

\def\bP{{\mathbb P}}

\def\bR{{\mathbb R}}
\def\bZ{{\mathbb Z}}

\def\bfP{{\mathbf P}}

\def\bfX{{\mathbf X}}

\def\H{\mathrm{H}}
\def\id{\mathrm{id}}

\def\Pic{\mathrm{Pic}}

\def\Hom{\mathrm{Hom}}
\def\Spec{\mathrm{Spec}}

\def\Pic{\mathrm{Pic}}
\def\GL{\mathrm{GL}}
\def\PGL{\mathrm{PGL}}

\def\Gal{\mathrm{Gal}}
\def\Aut{\mathrm{Aut}}

\def\Aff{\mathrm{Aff}}
\def\fod{\mathrm{FOD}}

\makeatother
\makeatletter

\author{Tianzhi Yang}
\address{Piazza dei Cavalieri, 7, 56126 Pisa, Italy}
\email{tianzhi.yang@sns.it}

\title{On the Fields of Moduli of Curves of Genus Six}

\begin{document}
\date{\today}
\begin{abstract}
The field of moduli of a variety $X$ over an algebraically closed field $K$ is defined as the fixed field of those automorphisms $\sigma$ of $K$ for which $X\cong X^{\sigma}$. A fundamental question is under what conditions a variety admits a model over its field of moduli. We investigate this problem for curves of genus $6$, by considering the stratification of the moduli space $M_6$. 
\end{abstract}

\maketitle
\setcounter{tocdepth}{2}

\section{Introduction}
We work over a field $k$ of characteristic $0$ with algebraic closure $K$. Let $(X,\xi)$ be a variety $X$ endowed with some additional structure $\xi$ over $K$, for instance a vector bundle, a polarization, a rational point, etc. Define the set of fields
\[
\fod=\{K/h/k \mid (X,\xi)\text{ admits a model over }h\},
\]
where $\fod$ stands for ``field of definition''. One may expect that there exists a minimal element in $\fod$, but this is in general false.

\subsection{Field of Moduli}
However, note that if $h\in \fod$, for every $\sigma\in \Gal(K/h)$, we have $(X,\xi)^{\sigma}\cong (X,\xi)$, where $(X,\xi)^{\sigma}$ is the Galois conjugation. So it is natural to make the following definition, which was first introduced by Matsusaka in \cite{Matsusaka1958}: let $\Delta\subset \Gal(K/k)$ be the subgroup that consists of elements $\sigma\in \Gal(K/k)$ with $(X,\xi)^{\sigma}\simeq (X,\xi)$. One can verify that $\Delta$ is an open subgroup, and the corresponding fixed field $k_{(X,\xi)}$ is called the \emph{field of moduli} of $(X,\xi)$. By definition, every element of $\fod$ contains $k_{(X,\xi)}$.

This naturally leads to the following question: does $k_{(X,\xi)}$ itself belong to $\fod$? Pioneering work on this topic was carried out by A. Weil \cite{Weil1956}, T. Matsusaka \cite{Matsusaka1958}, and G. Shimura \cite{Shimura1959}, and the question has continued to draw considerable attention; for example, in the setting of curves and abelian varieties, we may consult \cite{Murabayashi1996}, \cite{DebesDouai1997}, \cite{DebesEmsalem1999}, \cite{CardonaQuer2005}, \cite{Huggins2007}, \cite{Kontogeorgis2009}, \cite{Hidalgo2009}, \cite{Marinatto2013}. In addition, there is related work concerning Lie algebras; see \cite{demarche2026notecomplexliealgebras}.

More recently, G. Bresciani and A. Vistoli introduced a novel approach to this question using the theory of \emph{gerbes} \cite{BrescianiVistoli2024}. When the algebraic object $(X,\xi)$ has a finite automorphism group $\Aut(X,\xi)$, one can attach to it an algebraic gerbe $\cG_{(X,\xi)}$ over $k_{(X,\xi)}$, called the residue gerbe, which classifies twisted forms of $(X,\xi)$. The object $(X,\xi)$ descends to its field of moduli exactly when this gerbe is neutral, that is, when
\[
\cG_{(X,\xi)}(k_{(X,\xi)})\neq \emptyset .
\]
This perspective leads to several new applications. For example, in \cite{plane-curve}, G. Bresciani shows that a smooth complex plane curve $C$ of odd degree can be defined by a polynomial with real coefficients if and only if $C \simeq \bar{C}$. No elementary proof of this statement is known, and the corresponding claim fails in general for curves of even degree.

\subsection{Previous Results on Curves}
For curves of genus $0$, the problem is trivial. In the case of elliptic curves, it is a classical fact that the field of moduli is generated by the $j$-invariant, and furthermore that every elliptic curve admits a model over its field of moduli. Curves of genus $2$ are hyperelliptic; this situation is analyzed by Huggins in \cite{Huggins2007}, who showed that a hyperelliptic curve $C$ is not defined over its field of moduli only if $\Aut(C)/\langle \iota \rangle$ is a cyclic group, where $\iota\in \Aut(C)$ is the hyperelliptic involution. The result is later rephrased in terms of gerbes in \cite{bresciani-p1}. A genus $3$ curve is either hyperelliptic or a plane quartic, and the latter case is treated in \cite{plane-curve}.

These results rely on the fact that the residue gerbe $\cG_C$ admits a morphism to $B\PGL_n$. As an example, note that the automorphism group of a plane curve of degree $\geq 4$ is always a linear group. Thus one may take the algebraic structure $(X,\xi)$ above to be $(\bP^2_K, C)$; in this case the automorphism group $\Aut(\bP^2_K, C)$ is a finite subgroup of $\PGL_3$, and one checks that the residue gerbe $\cG_C$ associated with $C$ carries a faithful morphism to $B\PGL_3$. In general, let $\Gamma$ be an algebraic group defined over $k$, and consider a faithful morphism of gerbes $\cG \to B\Gamma$. In joint work with G. Bresciani, we constructed a framework to analyze when such a gerbe $\cG$ is neutral; see \cite{neutral-representation-dim-leq3}. We will make use of this framework in our study of fields of moduli of genus $6$ curves.

\subsection{Our Results} Let $C$ be a smooth curve of genus $6$ over $K$. The moduli space $M_6$ of genus $6$ curves admits the stratification
$$M_6 = U \sqcup Z_{<5} \sqcup Z_q \sqcup Z_h \sqcup Z_b \sqcup Z_t,$$
where $Z_q, Z_h, Z_b$, and $Z_t$ denote, respectively, the loci of smooth plane quintic curves, hyperelliptic curves, bielliptic curves, and trigonal curves. The open subset $U$ consists of those curves $C$ that are contained in a smooth del Pezzo surface $S_C$ of degree $5$, whereas $Z_{<5}$ parametrizes curves $C$ lying on a singular del Pezzo surface, in this case, the singularites on $S_C$ can only be one of the following ADE type
$$
    A_1,\quad 2A_1,\quad A_2,\quad A_2+A_1,\quad A_3,\quad A_4;
$$
see \cite{Zhao2024Moduli}.

Denote by $k_C$ the field of moduli of $C$, we first collect the three cases that are already known or partly known.

\begin{mainthm}
\label{thm:known-cases}
Let $C$ be a smooth curve of genus $6$ over $K$.
\begin{enumerate}[label=\textnormal{(\roman*)}]
    \item If $C$ is hyperelliptic and $\Aut(C)/\langle\iota\rangle$ is non-cyclic, then $C$ descends to $k_C$.
    \item If $C$ is a plane quintic, then $C$ descends to $k_C$.
    \item If $C$ is trigonal and does not descend to $k_C$, then $\Aut(C)$ is cyclic or a $C_3$-extension of a cyclic group.
\end{enumerate}
\end{mainthm}

As previously noted, the hyperelliptic case is due to Huggins \cite[Theorem~5.4]{Huggins2007}, and the statement for plane quintics follows from \cite[Theorem~1.1]{plane-curve}. For trigonal curves, one can appeal to \cite[Theorem~3.1 and Corollary~3.5]{ArtebaniQuispe2012}: a non-normal 3-gonal curve descends to its field of moduli, whereas in the normal situation the deck group is $C_3$ and the curve descends whenever the reduced automorphism group is non-cyclic. Note, however, that a trigonal genus $6$ curve is not necessarily 3-gonal in the sense of \cite{ArtebaniQuispe2012}. In \cite{BujalanceCosta2015}, the authors showed that the automorphism group of a $p$-gonal genus $g$ pseudoreal curve—i.e.\ a curve with field of moduli $\bR$ that nevertheless does not admit a model over $\bR$—is either cyclic or of the form $C_n\rtimes C_p$, under the hypothesis $g>(p-1)^2$ (which guarantees uniqueness of the pencil $g_p^1$). In Section~\ref{sect:proof}, we give a much simpler proof of the trigonal case. Our argument uses the uniqueness of the trigonal pencil together with the induced morphism between the associated residue gerbes, and it treats both Maroni types in a uniform way. Furthermore, this method generalizes the result of \cite{BujalanceCosta2015} to arbitrary base field $k$.

Our new results concern the remaining three strata.

\begin{mainthm}
\label{thm:new-cases}
Let $C$ be a smooth curve of genus $6$ over $K$.
\begin{enumerate}[label=\textnormal{(\roman*)}]
    \item If $C$ is bielliptic, then $C$ descends to $k_C$.
    \item If $C$ lies on a smooth quintic del Pezzo surface and does not descend to $k_C$, then either $\Aut(C)\cong C_2$, or $\Aut(C)\cong D_{10}$ and $\sqrt{-1}\notin k$.
    \item If the uniquely determined quintic del Pezzo surface $S_C$ is singular of type different from $2A_1$, then $C$ descends to $k_C$.
\end{enumerate}
\end{mainthm}

The proof is organized around distinguished structures and residue gerbes. For each case, we construct a faithful morphism
$$
    \cG_C\ra B\Gamma
$$
over $k_C$, where $\Gamma$ is the automorphism group of an distinguished ambient space or the linear group determined by an distinguished representation. Here by distinguished, we mean those additional algebraic structures on $C$ that is preserved by every Galois-semilinear automorphism; see section~\ref{gerbe-map-del-pezzo-case}. For the bielliptic case, the bielliptic involution splits the Hodge bundle into eigenspaces and gives $\Gamma=\bG_m\times\GL_5$. For the smooth del Pezzo surface case, the ambient surface gives us $\Gamma=S_5$, while for the singular del Pezzo surface case, we have $\Gamma=\Aut(S_\tau)$ for a split surface model $S_\tau$. 

We then use the methods of \cite{neutral-representation-dim-leq3} to study when the image of $\Aut(C)$ is a \emph{neutral subgroup} of $\Gamma(K)$. If $\Aut(C)$ is neutral, $C$ descends to its field of moduli. Section~\ref{sect:neutral-S_5} provides a classification of the neutral subgroups of $S_5$. The only exceptions are the double-transposition subgroup $C_2$, which is never neutral, and the transitive subgroup $D_{10}$, which is neutral precisely when $\sqrt{-1}$ lies in the base field. In the singular case, the automorphism groups described in \cite{Virin2024} are split solvable for the singularity types $A_4$, $A_3$, and $A_2+A_1$, we show that every finite subgroup of a connected split solvable group is neutral; the cases $A_2$ and $A_1$ call for additional normalizer computations. The sole remaining exceptional case is the type $2A_1$.

Under the guidance of \cite[Proposition~5.0.5]{Huggins2006PHD-thesis}, it is straightforward to find hyperelliptic curves of genus $6$ that are not defined over their fields of moduli. For trigonal case, such a counterexample also exists by \cite[Theorem 6]{BujalanceCosta2015}. In \cite{ARTEBANI20172383}, the authors showed that there exists a genus $6$ pseudoreal curve with automorphism group $D_{10}$. A closer look of their example shows that this curve belongs to the smooth del Pezzo locus $U\subset M_6$, so it agrees with our result since $\sqrt{-1}\notin \bR$. For the remaining cases, no counterexamples have yet been found, although it seems likely that some exist.

\subsection{Structure of the paper} In section~\ref{sect:preliminary}, we fix some terminology concerning gerbes and recall the definition of neutral subgroups. In section~\ref{sect:the-entire-section-classification}, we recall the stratification of genus $6$ curves. In section~\ref{sect:gerbes}, we find the desired distinguished structures and construct the morphisms $\cG\ra B\Gamma$. Section~\ref{sect:neutral-S_5} classifies the neutral subgroups of $S_5$. The main theorems are proved in section~\ref{sect:proof}; the analysis of the automorphism groups of singular quintic del Pezzo surfaces is included in the corresponding part of that section.

\subsection{Convention}
Throughout, we use $D_{2n}$ to denote the dihedral group with $2n$ elements.

\section{Preliminary on gerbes}\label{sect:preliminary}
In this section, we give a brief review of the basic properties of gerbes and of several results from \cite{neutral-representation-dim-leq3} that will be used later in the paper. This section does not contain any new material.
\subsection{Gerbes}
Recall that a gerbe $\cG$ over $k$ is a stack on the fppf site $(\Aff/k)$ of affine $k$-schemes which is locally non-empty and for which any two objects become isomorphic after passing to an fppf cover. Throughout this article, gerbes are assumed to be affine and of finite type; by definition, this means that the diagonal morphism is affine and of finite type, and that there exists an affine chart. Any such gerbe $\cG$ is automatically smooth, so there is a separable extension $k'/k$ with $\cG(k')\neq \emptyset$, see \cite[Corollary 3.2]{borne-vistoli}. A gerbe is said to be finite, étale, or unramified, respectively, if its inertia is finite, étale, or unramified, respectively, \cite[Tag 050P]{stacks-project}, \cite[Definition 8.1.17]{olsson}. 

The fundamental group of an affine gerbe $\cG$ is defined as the automorphism group of a geometric point $\Spec K \to \cG$. Since any two geometric points are isomorphic \cite[Proposition 3.1(c)]{borne-vistoli-nori}, this definition does not depend on the chosen point.

We say that a gerbe $\cG$ over $k$ is neutral if it has a $k$-rational point; in this situation, $\cG$ is isomorphic to the classifying stack $B\mathfrak{G}$ of some group scheme $\mathfrak{G}$, see \cite[Chapter III 2.2.6]{Giraud1971}.

A morphism $f:\cG\ra \cH$ of affine gerbes is faithful (resp. locally full), if for some field extension $h/k$ and some object $a\in \cG(h)$, the induced homomorphism $\underline{\Aut}_{\cG}(a)\ra \underline{\Aut}_{\cH}(f(a))$ of group schemes is a monomorphism (resp. faithfully flat), \cite[Definition 3.4]{borne-vistoli}.

\begin{lemm}\cite[Definition 3.8, Proposition 3.9]{borne-vistoli}\label{canonical-factorization}
    Let $\cG\ra \cH$ be a morphism of affine gerbes, there exists a factorization $\cG\ra \cG'\ra \cH$ such that $\cG\ra \cG'$ is locally full and $\cG'\ra \cH$ is faithful, and if $\cG\ra \cG''\ra \cH$ is another such factorization, there exists an equivalence $\cG'\ra \cG''$ making the obvious diagram commute.
\end{lemm}

The language of gerbes is particularly well suited to studying the relationship between field of moduli and field of definition. Concretely, saying that the pair $(X,\xi)$ admits a field of definition over an extension $h/k$ means that there exists a twisted form of $(X,\xi)$ over $h$, where $(X,\xi)$ denotes a variety $X$ over $K$ equipped with some additional structure, as specified in the introduction. Hence the problem reduces to understanding twisted forms of $(X,\xi)$, and these are classified in terms of gerbes, see \cite[\S 5]{BrescianiVistoli2024}, \cite[\S 2]{bresciani-p1}. 

The gerbe $\cG_{(X,\xi)}$ that classifies twisted forms of $(X,\xi)$ is called the residue gerbe of $(X,\xi)$. It is an algebraic stack over $k$, and an fppf-gerbe over a field extension $k_{(X,\xi)}/k$, see \cite[Proposition 3.10]{BrescianiVistoli2024}. This field $k_{(X,\xi)}$ is precisely the field of moduli of $(X,\xi)$. In this framework, the statement that $(X,\xi)$ has a field of definition $h/k$ is equivalent to $\cG_{(X,\xi)}(h)\neq \emptyset$. Moreover, after base changing to $k=k_{(X,\xi)}$, the pair is defined over its field of moduli if and only if $\cG_{(X,\xi)}$ is neutral. Thus, identifying conditions under which this gerbe is neutral is crucial.

\subsection{Neutral Subgroups}\label{sect:neutral-subgroups} Let $\Gamma$ be an affine group scheme of finite type over $k$, and let $j:G\hookrightarrow \Gamma(K)$ be a finite subgroup, assumed to be stable under the Galois action.

\begin{defi}\cite[Definition 3.3]{neutral-representation-dim-leq3}\label{def:ggerbe}
    A \emph{$j$-gerbe} is a faithful morphism of gerbes $\cG\ra B\Gamma$ over some field extension $k'/k$, such that the image in $\Gamma$ of the automorphism group of one (hence of every) geometric point of $\cG$ is conjugate to $G$ inside $\Gamma(K)$.
\end{defi}

We call a finite subgroup $j:G\hookrightarrow \Gamma(K)$ \emph{neutral} if, for every $j$-gerbe $\cG\ra B\Gamma$ over some extension $k'/k$, one has $\cG(k')\neq \emptyset$; see \cite[Definition 3.4]{neutral-representation-dim-leq3}. Thus, neutrality is a property of the specific embedding $j$, rather than of the abstract group $G$ alone. Key examples of such a group scheme $\Gamma$ can be $\GL_n,\PGL_n,S_n$.

To find neutral subgroups of $\Gamma(K)$, it can be convenient to replace $\Gamma$ by a suitable smaller subgroup.

\begin{defi}\cite[Definition 4.1]{neutral-representation-dim-leq3}\label{def:gate}
    Let $j:G\hookrightarrow \Gamma(K)$ be a finite subgroup. A subgroup scheme $\Pi\subset \Gamma$ over $k$ is called a \emph{gate} for $G$ if $G\subset \Pi(K)$ and, for every $j$-gerbe $\cG\ra B\Gamma$ over some extension $k'/k$, there exists a factorization $\cG\ra B\Pi \ra B\Gamma$ such that the induced map $\cG\ra B\Pi$ is a $j'$-gerbe with respect to the embedding $G\subset \Pi(K)$.
\end{defi}
Although finding gates can be hard, any finite subgroup $G\subset \Gamma(K)$ always has a gate available by default, namely its normalizer.

\begin{prop}\cite[Proposition 4.13]{neutral-representation-dim-leq3}\label{prop:normalizer}
    Let $\mathfrak{G}\subset \Gamma$ be a finite subgroup scheme. Then the normalizer of $\mathfrak{G}$ is a gate for $\mathfrak{G}(K)\subset \Gamma(K)$.
\end{prop}

\begin{coro}
    If $G\subset \Gamma(K)$ is self-normalizing and Galois invariant, then $G$ is a neutral subgroup of $\Gamma(K)$.
\end{coro}
\begin{proof}
    Let $\cG\ra B\Gamma$ be a $j$-gerbe over $k'$. By \ref{prop:normalizer}, there is a factorization
    $$\cG\ra B\mathfrak{G}\ra B\Gamma,$$
    where $\cG\ra B\mathfrak{G}$ is a twisted form of $BG\ra BG$. The latter is an isomorphism, hence $\cG\ra B\mathfrak{G}$ is already an isomorphism, and thus $\cG$ is neutral.
\end{proof}

\begin{theo}\cite[Theorem 4.18]{neutral-representation-dim-leq3}\label{thm:coh-crit}
    Let $\Pi\subset \Gamma$ be a gate for $G\subset \Gamma(K)$, and assume that $\mathfrak{G}$ is normal in $\Pi$. Set $\mathfrak{Q}=\Pi/\mathfrak{G}$. Then $G$ is neutral if and only if
    $$\H^1(k',\Pi)\ra \H^1(k',\mathfrak{Q})$$
    is surjective for every field extension $k'/k$.
\end{theo}

\begin{coro}\cite[Corollary~4.20]{neutral-representation-dim-leq3}
\label{coro:special-normalizer}
Let $\mathfrak{N}\subset\Gamma$ be the normalizer of $\mathfrak{G}$. If $\mathfrak{N}/\mathfrak{G}$ is special, then $G$ is neutral.
\end{coro}

\section{The classification of smooth genus six curves}
\label{sect:the-entire-section-classification}
Let $C$ be a smooth projective curve of genus $6$ defined over $K$. The goal of this section is to specify the geometric stratification that will be used in the remainder of the paper.

\subsection{Geometric classification}\label{sect:classification}
For a curve $C$ as above, exactly one of the following holds.
\begin{enumerate}[label=\textnormal{(\roman*)}]
    \item $C$ is hyperelliptic.
    \item $C$ is trigonal. Its canonical model lies on the rational normal scroll determined by the unique $g^1_3$.
    \item $C$ is a smooth plane quintic. Its canonical model is the image of its plane model under the Veronese embedding
    \[
        v_2\colon \bP^2\hookrightarrow \bP^5.
    \]
    \item $C$ is bielliptic. Its canonical model is a quadric section of the cone in $\bP^5$ over an elliptic normal quintic in $\bP^4$.
    \item $C$ does not fall into any of the previous classes. In this case, its canonical model is a quadric section
    \[
        C\in |-2K_{S_C}|
    \]
    of a unique quintic del Pezzo surface $S_C\subset \bP^5$ having at worst rational double points.
\end{enumerate}
In situation \textnormal{(v)}, the surface $S_C$ is either smooth or an ADE quintic del Pezzo surface. In the singular case its type is one of
$$
    A_1,\quad 2A_1,\quad A_2,\quad A_2+A_1,\quad A_3,\quad A_4.
$$

\paragraph{\textbf{References and explanation:}}
The fact that the above list is exhaustive is obtained by combining the special and non-special cases. The special curves of genus six are exactly the smooth plane quintics, hyperelliptic curves, trigonal curves, and bielliptic curves; see \cite[Section~2.3]{Zhao2024Moduli}. If $C$ is not special, then by \cite[Claim~5.14]{ArbarelloHarris1981} its canonical model lies on the unique quintic del Pezzo surface described in \textnormal{(v)}, so there are no additional possibilities. 

The scroll description of a trigonal genus-six curve and its two possible Maroni types are given in \cite[Proposition~2.17]{Zhao2024Moduli}. The uniqueness of the trigonal pencil for genus at least $5$ and of the bielliptic involution for genus at least $6$ follows from the Castelnuovo–Severi inequality; see, for instance, \cite[p.~2]{Schweizer2015}. The description of the bielliptic canonical model as a quadric section of a cone is stated in \cite[Section~7.3]{Zhao2024Moduli}. Finally, the six possible ADE types of singular quintic del Pezzo surfaces arise from the classification in \cite[Theorem~3.4]{HidakaWatanabe1981}, and an explicit treatment in the genus-six setting appears in \cite[Theorem~2.18]{Zhao2024Moduli}.

\subsection{Algebraic structures attached to the curve}
\label{sect:classification-detail}
Now we provide some detail for the classification in section~\ref{sect:classification}, which will be used later. Recall that
\[
    W^1_3(C)
    :=
    \bigl\{[A]\in\Pic^3(C)\mid h^0(C,A)\geq 2\bigr\}
\]
is the Brill--Noether locus of degree $3$ line bundles with at least two independent
sections, endowed with its usual determinantal scheme structure.  A base-point-free point
$[A]\in W^1_3(C)$ with $h^0(C,A)=2$ is a $g^1_3$, and its complete linear series gives
\[
    f_A\colon C\ra \bP\bigl(\H^0(C,A)^\vee\bigr)\simeq \bP^1_K.
\]
For a trigonal curve of genus $6$, this pencil is unique, so the morphism is intrinsic up to
the ambiguity of choosing projective coordinates on its target.

The canonical scroll is isomorphic to $\bF_0$ or $\bF_2$.  If $E$ denotes the negative section and $F$ a fibre, the curve has class
\[
    3E+4F\text{ on }\bF_0,
    \quad
    3E+7F\text{ on }\bF_2.
\]

A curve is bielliptic if it admits a morphism of degree $2$ to a smooth curve of
genus $1$.  Thus there is a finite morphism $\pi: C\ra E$. Because the characteristic is $0$, the quadratic extension $K(C)/K(E)$ is separable and
hence Galois.  Its nontrivial deck transformation is an involution $\tau\in\Aut(C)$, and
$E=C/\langle\tau\rangle$.  Conversely, an involution whose quotient has genus $1$ gives a
bielliptic structure.

For genus $6$ this involution is unique.  Indeed, if $C$ admitted two distinct degree-$2$
maps $\pi_i\colon C\ra E_i$ to genus-one curves, the Castelnuovo--Severi inequality would
give
\[
    g(C)\leq 2g(E_1)+2g(E_2)+(2-1)(2-1)=5,
\]
a contradiction.  See \cite[Chapter~VIII, Exercises~C-1 and~C-2]{ACGH1985} for the
Castelnuovo--Severi argument. Uniqueness also implies centrality: for every $u\in\Aut(C)$, the conjugate $u\tau u^{-1}$ is again a bielliptic involution, hence equals $\tau$. Therefore both the quotient $E$ and the map $\pi$ are intrinsic.

We now make the structure of the bielliptic double cover explicit. Since a non-constant morphism between projective curves is finite, $\pi$ is finite of degree $2$. Moreover, $\pi$ is flat: indeed, $\pi_*\mathcal O_C$ is a torsion-free sheaf of rank
$2$ on the smooth curve $E$, and is therefore locally free.

The deck transformation $\tau$ acts $\mathcal O_E$-linearly on $\pi_*\mathcal O_C$. This action decomposes $\pi_*\mathcal O_C$ into the eigensheaves corresponding to the two characters of $C_2$. The invariant subsheaf is $\mathcal O_E$, since $E$ is the quotient
$C/\langle\tau\rangle$. The anti-invariant subsheaf is therefore a line bundle; writing it as $L^{-1}$, we obtain
\[
\pi_*\mathcal O_C\simeq \mathcal O_E\oplus L^{-1}.
\]
This is the case $G=C_2$ of the character decomposition for abelian covers
in \cite[Equation (1.1)]{Pardini1991}.

The only non-trivial part of the multiplication on $\pi_*\cO_C$ is therefore a morphism
\[
L^{-1}\otimes L^{-1}\to\mathcal O_E,
\]
or, equivalently, a section $s\in \H^0(E,L^{\otimes 2})$. Thus
\[
C\simeq
\operatorname{Spec}_E\bigl(\mathcal O_E\oplus L^{-1}\bigr),
\]
where the multiplication on $L^{-1}\otimes L^{-1}$ is given by $s$.
If $B=(s)_0$ denotes the zero divisor of $s$, then $L^{\otimes 2}\simeq\mathcal O_E(B)$. This is precisely the degree $2$ case of Pardini's building-data
description of abelian covers; see
\cite[Definition~2.1 and Theorem~2.1, p.~196]{Pardini1991}. Conversely, the
pair $(L,s)$ reconstructs the double cover.

Locally on $E$, after trivializing $L$, the cover is given by an equation
\[
z^2=f
\]
for a regular function $f$ representing $s$.  A point belongs to the
branch locus precisely when $f$ vanishes there.  If $f$ had a zero of
order at least $2$, then at the corresponding point of $C$ both
derivatives of $z^2-f$ would vanish, contradicting the smoothness of
$C$.  Hence $B$ is reduced. 

Riemann--Hurwitz now gives
\[
2g(C)-2
 =
2\bigl(2g(E)-2\bigr)+\deg B.
\]
Since $g(C)=6$ and $g(E)=1$, it follows that $\deg B=10$. The relation $L^{\otimes 2}\simeq\mathcal O_E(B)$ therefore implies $\deg L=5$.

The canonical bundle of a double cover defined by $(L,s)$ is
\[
\omega_C\simeq\pi^*(\omega_E\otimes L).
\]
To keep track of the action of
$\tau$ on canonical bundle, it is convenient to use finite
flat duality:
\[
\begin{aligned}
\pi_*\omega_C
&\simeq
\mathcal Hom_E(\pi_*\mathcal O_C,\omega_E)\\
&\simeq
\mathcal Hom_E(\mathcal O_E\oplus L^{-1},\omega_E)\\
&\simeq
\omega_E\oplus(\omega_E\otimes L).
\end{aligned}
\]
The first summand is $\tau$-invariant and the second is
$\tau$-anti-invariant (since $L^{-1}$ is the non-trivial eigensheaf). Since $E$ is elliptic,
$\omega_E\simeq\mathcal O_E$, and hence
\[
\H^0(C,\omega_C)^+
 \simeq \H^0(E,\mathcal O_E),
\quad
\H^0(C,\omega_C)^-
 \simeq \H^0(E,L).
\]
In particular,
\[
\dim \H^0(C,\omega_C)^+=1,
\quad
\dim \H^0(C,\omega_C)^-=5,
\]
where the second equality follows from Riemann--Roch on $E$ and
$\deg L=5$.

A line bundle of degree at least $3$ on an elliptic curve is very ample.
Consequently, the complete linear series $|L|$ embeds $E$ as an elliptic
normal quintic
\[
E\hookrightarrow
\bP\bigl(\H^0(E,L)^\vee\bigr)\simeq\bP^4.
\]
Let
\[
v=
\bP\bigl(\H^0(C,\omega_C)^{+,\vee}\bigr)
\subset
\bP\bigl(\H^0(C,\omega_C)^\vee\bigr)\simeq\bP^5
\]
be the point corresponding to the one-dimensional invariant
eigenspace. Projection of the canonical model of $C$ from $v$ is the
morphism defined by the anti-invariant canonical differentials.
It factors as
\[
C\xrightarrow{\ \pi\ }E
 \xrightarrow{\ |L|\ }\mathbb P^4.
\]
It follows that the canonical image of $C$ is contained in the cone
with vertex $v$ over the elliptic normal quintic $E\subset\mathbb P^4$.

\section{Residue gerbes}
\label{sect:gerbes}

In this section, we investigate the residue gerbe $\cG_C$ of $C$. Our main objective is to construct an algebraic group scheme $\Gamma$ over $k_C$, naturally attached to some additional structure of $C$, in such a way that there is a faithful morphism $\cG_C \to B\Gamma$.

\subsection{del Pezzo case}\label{gerbe-map-del-pezzo-case}
Suppose $C$ lies on a del Pezzo surface. Let $\mathcal M_6$
be the moduli stack of smooth curves of genus $6$.  We write
\[
    \cG_C\subset \mathcal M_6
\]
for the residue gerbe of the geometric point $[C]$; it is an fppf gerbe over $k_C$, and its
geometric inertia group is $\Aut(C)$; see \cite[Proposition 3.10]{BrescianiVistoli2024}.  

Let $\mathcal D_6$ be the moduli stack of pairs
$(S,D)$ in which $S$ is a quintic del Pezzo surface with at worst rational double points
and $D\in |-2K_S|$ is a smooth curve.  For a non-special curve, we write
\[
    \cG_{(S_C,C)}\subset \mathcal D_6
\]
for the residue gerbe of the pair $(S_C,C)$; it is defined over the field of moduli
$k_{(S_C,C)}$ and has geometric inertia group $\Aut(S_C,C)$.  Equivalently, these gerbes
encode the forms of $C$ and of $(S_C,C)$, respectively, together with their isomorphisms.

The uniqueness of the del Pezzo surface in case \textnormal{(v)} of the classification list \ref{sect:classification} identifies the two gerbes just defined. Recall that for $\sigma\in \Gal(K/k_C)$, a $\sigma$-semilinear isomorphism of $C$ is an isomorphims $C^{\sigma}\to C$ over $K$.

\begin{prop}
\label{prop:intrinsic-del-pezzo}
Let $C$ be a non-special smooth curve of genus $6$. Every $\sigma$-semilinear isomorphism $C^\sigma\simeq C$ extends uniquely to an isomorphism of pairs
\[
    (S_C^\sigma,C^\sigma)\simeq (S_C,C).
\]
In particular,
\[
    k_C=k_{(S_C,C)},
    \quad
    \cG_C\simeq \cG_{(S_C,C)},
    \quad
    \Aut(C)=\Aut(S_C,C).
\]
\end{prop}

\begin{proof}
Note that there is a natural forgetful morphism $\cG_{(S_C,C)} \to \cG_C$, and by definition of field of moduli, we have $k_C\subset k_{(S_C,C)}$. So it is enough to show $k_{(S_C,C)}\subset k_C$ and $\Aut(C)=\Aut(S_C,C)$, since the isomorphism of residue gerbes can be checked after base change to $K$.

Let $\varphi: C^{\sigma} \xrightarrow{\sim} C$ be an isomorphism over $K$. Pullback of the canonical divisor induces an
isomorphism
\[
\varphi^{*}: \H^{0}(C,\omega_{C})
\xrightarrow{\sim}
\H^{0}(C^{\sigma},\omega_{C^{\sigma}}).
\]
Hence, the diagram
\begin{center}
    \begin{tikzcd}
C^{\sigma} \arrow[r, "\varphi"] \arrow[d] & C \arrow[d] \\
\bP^5 \arrow[r, "\Phi_{\varphi}"]         & \bP^{5}    
\end{tikzcd}
\end{center}
commutes, where the two vertical maps are canonical embeddings.

In particular, $\Phi_{\varphi}(S_C^{\sigma})$ is a quintic del Pezzo surface
containing the image of $C$ under the canonical embedding. By the uniqueness of the quintic
del Pezzo surface stated in \ref{sect:classification}, one has
\[
\Phi_{\varphi}(S_C^{\sigma})=S_C.
\]
Therefore the restriction
\[
\tilde{\varphi}
:=
\Phi_{\varphi}|_{S_C^{\sigma}}
\colon
S_C^{\sigma}
\xrightarrow{\sim}
S_C
\]
is an isomorphism extending $\varphi$. Equivalently, the diagram
\begin{center}
    \begin{tikzcd}
C^{\sigma} \arrow[r, "\varphi"] \arrow[d] & C \arrow[d] \\
S_C^{\sigma} \arrow[r, "\tilde{\varphi}"] & S_C        
\end{tikzcd}
\end{center}
commutes.

The extension is unique. Indeed, an automorphism of $S_C^{\sigma}$ whose restriction to $C^{\sigma}$ is the identity induces a projective transformation of the anticanonical space that fixes the canonical curve $C^\sigma$ pointwise. Since the canonical curve spans $\bP^5_K$, this projective transformation is the identity, and hence so is the automorphism of $S_C^\sigma$.

It follows that $\Aut(C)\cong \Aut(S_C,C)$ and $k_C=k_{(S_C,C)}$. 
\end{proof}

\begin{prop}
\label{prop:del-pezzo-marking-morphism}
Suppose that $S_C$ is smooth. Then there is a faithful morphism
$$
    \rho:\cG_C\to BS_5,
$$
where $S_5$ denotes the constant group scheme over $k_C$.
\end{prop}

\begin{proof}
Set $k_0=k_{(S_C,C)}$. We first choose a split model of the quintic del Pezzo surface. Let
\[
S_0
:=
\operatorname{Bl}_{p_1,p_2,p_3,p_4}\bP^2_{k_0},
\]
where
\[
p_1=(1:0:0),\qquad
p_2=(0:1:0),\qquad
p_3=(0:0:1),\qquad
p_4=(1:1:1).
\]
These four points are in general position, so $S_0$ is a smooth
quintic del Pezzo surface.

We recall the structure of the automorphism group of $S_0$. The surface $S_0$ contains ten $(-1)$-curves: the four exceptional divisors above the points $p_i$, together with the strict transforms of the six lines joining $p_i$ and $p_j$. These curves form the Petersen graph. Because all the points $p_i$ are $k_0$-rational, the Galois group $\Gal(K/k_0)$ acts trivially on this Petersen graph. By \cite[Proposition 3.8]{Boi26}, it follows that
\[
\Aut_{k_0}(S_0)\simeq S_5.
\]

Over an algebraically closed field, any smooth del Pezzo surface of degree $5$ is isomorphic to the blow-up of $\bP^2$ at four points in general position, and such a surface is unique up to isomorphism. In particular, we have $S_{0,K}\simeq S_C$.

Let now $T\in (\Aff/k)$ and consider an object
\[
\xi=(S,D)\in\cG_{(S_C,C)}(T).
\]
Since $\cG_{(S_C,C)}$ is the residue gerbe of $(S_C,C)$ and $S_{0,K}\simeq S_C$, the surface $S$ is fppf-locally on $T$ isomorphic to $S_{0,T}$. Consider the fppf sheaf
\[
\cP_\xi
:=
\underline{\mathrm{Isom}}_T(S_{0,T},S),
\]
the constant group scheme
$S_5=\Aut_T(S_{0,T})$ acts on the right on $\cP_\xi$ by precomposition. It is straightforward to check that $\cP_\xi$ is an $S_5$-torsor over $T$ in the fppf topology. Thus the assignment $\xi\mapsto \cP_\xi$ defines a morphism of stacks
\[
\rho:\cG_{(S_C,C)}\to BS_5.
\]

We now show that $\rho$ is faithful. Let $g\in\Aut_T(S,D)$ be an automorphism whose image under $\rho$ is the identity automorphism of the torsor $\cP_\xi$. After passing to an fppf cover $T'\to T$, choose a section
\[
\alpha\in
\underline{\mathrm{Isom}}_{T'}(S_{0,T'},S_{T'}).
\]
The automorphism of $\cP_{\xi,T'}$ determined by $g$ maps $\alpha$ to $g_{T'}\circ\alpha$. Since this automorphism is trivial, we have $g_{T'}\circ\alpha=\alpha$. As $\alpha$ is an isomorphism, it follows that $g_{T'}=\id_{S_{T'}}$.

Because equality of morphisms can be checked after a faithfully flat base change, we conclude that $g=\id_S$. Hence, for every object $\xi$, the natural homomorphism
\[
\Aut_T(\xi)\to
\Aut_T\bigl(\rho(\xi)\bigr)
\]
is injective. This shows that $\rho$ is faithful.
\end{proof}

\begin{rema}
    The preceding proposition essentially shows the existence of a morphism from $\cG_C$ to the residue gerbe $\cH$ of smooth quintic del Pezzo surfaces. $\cH$ is neutral because there is a split model over $k_0$, namely $\operatorname{Bl}_{p_1,p_2,p_3,p_4}\bP^2_{k_0}$. Consequently, we obtain an identification $\cH\simeq BS_5$. If, instead, we take a non-split surface as the tautological section of $\cH$, then we obtain an equivalence $\cH\simeq B\mathfrak{H}$, where $\mathfrak{H}$ is a twisted form of the constant group scheme $S_5$.
\end{rema}

\begin{prop}
\label{prop:singular-del-pezzo-marking-morphism}
Suppose that $S_C$ is singular of ADE type $\tau$. Choose a split quintic del Pezzo surface $S_\tau$ over $k_C$ of type $\tau$, and set
$$
    \Gamma_\tau:=\Aut_{k_C}(S_\tau).
$$
Then there is a faithful morphism
$$
    \rho_\tau:\cG_C\to B\Gamma_\tau.
$$
\end{prop}

\begin{proof}
Split models for the six possible ADE types are listed in \cite[Table~3, cases $36^\circ$--$41^\circ$ and Appendix A for case $38^{\circ}$]{CheltsovProkhorov2021}. By Proposition~\ref{prop:intrinsic-del-pezzo}, we may again identify $\cG_C$ with $\cG_{(S_C,C)}$. For $\xi=(S,D)\in\cG_{(S_C,C)}(T)$, the fppf sheaf
$$
    \underline{\mathrm{Isom}}_T(S_{\tau,T},S)
$$
is a $\Gamma_\tau$-torsor. As in the proof of Proposition~\ref{prop:del-pezzo-marking-morphism}, this construction is functorial and faithful, and therefore gives the required morphism.
\end{proof}

\subsection{Bielliptic case}\label{sect:bielliptic-gerbe-construction} Suppose now $C$ is bielliptic, Let $p:\cC\ra \cG_C$ be the universal curve over $\cG_C$. Concretely, if $S$ is an affine $k$-scheme and $\xi:S\to\cG_C$ is an object of the gerbe, then the base change
\[
  p_\xi:\cC_\xi:=\cC\times_{\cG_C,\xi}S\to S
\]
is the smooth proper genus $6$ family represented by $\xi$. 

The \emph{Hodge bundle} on $\cG_C$ is $\bE:=p_*\omega_{\cC/\cG_C}$. For every $\xi:S\to\cG$, cohomology and base change give a canonical
identification
\[
  \xi^*\bE\simeq
  (p_\xi)_*\omega_{\cC_\xi/S}.
\]
The relative dualizing sheaf is invertible and compatible with arbitrary base change, and relative duality identifies
\[
  (p_\xi)_*\omega_{\cC_\xi/S}
  \simeq
  (R^1(p_\xi)_*\cO_{\cC_\xi})^{\vee}.
\]
Consequently, $\bE$ is locally free of rank $6$, its formation commutes with base
change, and its fibre at the geometric point of $\cG_C$ corresponding to $C$ is $\bE|_C\simeq \H^0(C,\omega_C)$. For the relative dualizing sheaf and its base-change property, see \cite[\href{https://stacks.math.columbia.edu/tag/0E6N}{Tag 0E6N}]{stacks-project}; for local freeness and base change of
$R^1p_*\cO$ in a family of curves, see \cite[\href{https://stacks.math.columbia.edu/tag/0GKA}{Tag 0GKA}]{stacks-project}.

The bielliptic involution $\tau$ is unique and central in $G=\Aut(C)$ as shown in section~\ref{sect:classification-detail}, so it is preserved by every arrow in the residue gerbe. Equivalently, the automorphism $\tau$ of the constant family glues to an automorphism of the universal curve $\cC/\cG_C$. It therefore acts on $\bE$. We can then split the Hodge bundle canonically
\[
  \bE=\bE_+\oplus\bE_-,
\]
just as in the case of the curve $C$.

The displayed summands are respectively the $(+1)$- and $(-1)$-eigensheaves, meaning that the involution $\tau$ acts as $(+1)$ and $(-1)$, respectively.
Therefore
\[
  \rk(\bE_+)=1,\quad \rk(\bE_-)=5,
\]
and their geometric fibers are
\[
  (\bE_+)_C=\H^0(E,\omega_E),
  \quad
  (\bE_-)_C=\H^0(E,\omega_E\otimes L).
\]

\begin{prop}\label{prop:faithful-hodge}
The two Hodge eigen sub-bundles determine a canonical faithful morphism
\[
  \Phi:\cG_C\to
  B(\bG_m\times\GL_5).
\]
On inertia group $G=\Aut(C)$ of the geometric point corresponding to $C$, the induced homomorphism is
\[
  \chi_+\times\rho_-:
  G\to
  \bG_m(K)\times
  \GL(\H^0(E,\omega_E\otimes L)),
\]
where $\chi_+$ is the action on $\H^0(E,\omega_E)$ and $\rho_-$ is the action on the
anti-invariant Hodge space.
\end{prop}

\begin{proof}
A rank $r$ vector bundle $\cV$ on a stack determines a morphism to
$B\GL_r$. Applying this construction to the line bundle $\bE_+$ and the rank $5$ bundle
$\bE_-$ gives the morphism. After evaluating at the geometric object $C$, its inertia
homomorphism is the product representation $\chi_+\times\rho_-$ displayed above.

It remains to prove faithfulness, it is enough to prove $\chi_+\times\rho_-$ is injective. Suppose $u\in G$ lies in the kernel. Then $u$ acts
trivially on both $\bE_+|_C$ and $\bE_-|_C$, hence trivially on their direct sum
\[
  \bE|_C=\H^0(C,\omega_C).
\]
But it is known that the canonical representation of the automorphism group of a
smooth curve of genus at least $2$ on $\H^0(C,\omega_C)$ is faithful; see
\cite[Theorem 3.2]{K_ck_2015}.
\end{proof}

\subsection{Hyperelliptic and trigonal} The remaining situations are simpler; we begin by assuming that $C$ is hyperelliptic or trigonal. In this setting, $C$ possesses a canonical $g_2^1$ or a canonical $g_3^1$, which gives a canonical map $f:C\to \bP^1_K$, up to automorphism of $\bP^1_K$. Fix such an map $f$, for any $\sigma$-semilinear isomorphism $\varphi: C^{\sigma}\to C$, $f \circ\varphi$ defines a morphism $C^{\sigma}\to \bP^1_K$. By the uniqueness of the pencil, there exists a unique isomorphism $\bP_K^{1,\sigma}\to \bP_K^1$ making the following diagram commute.
\begin{center}
   \begin{tikzcd}
C^{\sigma} \arrow[r] \arrow[d] & C \arrow[d] \\
{\bP^{1,\sigma}_K} \arrow[r]   & \bP^1_K    
\end{tikzcd} 
\end{center}

Then, we may consider the residue gerbe of the pair $(C,f)$, and a similar argument as in the proof of \ref{prop:intrinsic-del-pezzo} and \ref{prop:del-pezzo-marking-morphism} would produce a morphism $\cG_C\to B\PGL_2$.

\subsection{Plane quintic} Finally, assume that $C$ is a plane quintic, so there is an embedding $v:C\hookrightarrow \bP^2_K$ and we regard $C$ as a curve in $\bP^2_K$. The associated line bundle $\cO_C(1)$ is intrinsically determined by $C$, in the sense that any two embeddings of $C$ into $\bP^2_K$ differ by an automorphism of $\bP^2_K$; see \cite[Appendix A, \S 1, Exercise 18]{ACGH1985}.

Consequently, by the same reasoning as in the proof of Proposition~\ref{prop:intrinsic-del-pezzo}, the field of moduli of $C$ coincides with the field of moduli of the pair $(\bP^2_K,C)$, and their residue gerbes agree. The isomorphism torsor construction used in Proposition~\ref{prop:del-pezzo-marking-morphism} then gives a morphism $\cG_C\ra B\PGL_3$; see \cite[Section 5]{plane-curve}.

\section{\texorpdfstring{Neutral subgroups of $S_5$}{Neutral subgroups of S5}}
\label{sect:neutral-S_5}
In this section, we take a brief detour from the main thread of the paper to classify the neutral subgroups of $S_5$, the result will be applied to the proof of the main theorem. 

First, we enumerate all subgroups of $S_5$ up to conjugacy:
\[
\begin{aligned}
C_2^t&=\langle(12)\rangle,
& C_2^d&=\langle(12)(34)\rangle,\\
V_4^{\mathrm{tr}}&=\langle(12),(34)\rangle,
& V_4^K&=\{1,(12)(34),(13)(24),(14)(23)\},\\
S_3^{\mathrm{fix}}&=\langle(123),(12)\rangle,
& S_3^{\mathrm{tw}}&=\langle(123),(12)(45)\rangle.
\end{aligned}
\]
Also, define
\[
    F_{20}=C_5\rtimes C_4
    =\langle r,a\mid r^5=a^4=1,\ ara^{-1}=r^2\rangle,
\]
and we adopt the convention that $D_{2n}$ denotes the dihedral group of order $2n$. The superscripts on the subgroup names are mnemonic for their embeddings: “t” and “d” refer to subgroups generated by a single transposition and by a double transposition, respectively; “tr” indicates the Klein four group generated by transpositions; “K” refers to the normal Klein four subgroup of $S_4$; and “fix” and “tw” distinguish the embedding of $S_3$ that fixes a point from the sign-twisted embedding, respectively.

Now we can list all finite subgroups up to conjugation and their normalizers. 
\begin{prop}
\label{prop:s5-normalizer-table}
The group $S_5$ has nineteen conjugacy classes of subgroups. Table~\ref{tab:s5-normalizers} lists representatives for these classes together with their normalizers. A complement to $G$ in $N$ is a subgroup $H\subset N$ such that $H\cap G=\{1\}$ and $H\cong N/G$. So a dash in the final column indicates that the corresponding normalizer extension does not split.
\end{prop}

\begin{table}[ht]
\centering
\caption{Normalizer extensions for subgroups of $S_5$}
\label{tab:s5-normalizers}
\resizebox{\textwidth}{!}{%
\begin{tabular}{llll}
\toprule
$G$ & $N_{S_5}(G)$ & $N/G$ & A complement to $G$ in $N$\\
\midrule
$1$ & $S_5$ & $S_5$ & $S_5$\\
$C_2^t$ & $C_2\times S_{\{3,4,5\}}$ & $S_3$ & $S_{\{3,4,5\}}$\\
$C_2^d$ & $D_8$ & $V_4$ & --\\
$C_3=\langle(123)\rangle$ & $S_{\{1,2,3\}}\times S_{\{4,5\}}$ & $V_4$ & $\langle(12),(45)\rangle$\\
$V_4^{\mathrm{tr}}$ & $S_2\wr S_2$ on $\{1,2,3,4\}$ & $C_2$ & $\langle(13)(24)\rangle$\\
$C_4=\langle(1234)\rangle$ & $D_8$ & $C_2$ & $\langle(13)\rangle$\\
$V_4^K$ & $S_{\{1,2,3,4\}}$ & $S_3$ & $S_{\{1,2,3\}}$\\
$C_5=\langle(12345)\rangle$ & $F_{20}$ & $C_4$ & $\langle(2354)\rangle$\\
$S_3^{\mathrm{fix}}$ & $S_{\{1,2,3\}}\times S_{\{4,5\}}$ & $C_2$ & $\langle(45)\rangle$\\
$C_6=\langle(123),(45)\rangle$ & $S_{\{1,2,3\}}\times S_{\{4,5\}}$ & $C_2$ & $\langle(12)\rangle$\\
$S_3^{\mathrm{tw}}$ & $S_{\{1,2,3\}}\times S_{\{4,5\}}$ & $C_2$ & $\langle(45)\rangle$\\
$D_8=\langle(1234),(13)\rangle$ & $D_8$ & $1$ & $1$\\
$D_{10}=\langle(12345),(25)(34)\rangle$ & $F_{20}$ & $C_2$ & --\\
$S_3\times C_2$ & $S_3\times C_2$ & $1$ & $1$\\
$A_4$ on $\{1,2,3,4\}$ & $S_{\{1,2,3,4\}}$ & $C_2$ & $\langle(12)\rangle$\\
$F_{20}$ & $F_{20}$ & $1$ & $1$\\
$S_4$ fixing $5$ & $S_4$ & $1$ & $1$\\
$A_5$ & $S_5$ & $C_2$ & $\langle(12)\rangle$\\
$S_5$ & $S_5$ & $1$ & $1$\\
\bottomrule
\end{tabular}}
\end{table}
\begin{proof}
    A classification of these subgroups is provided in \cite[Appendix, \S2, pp.~198--200]{Swallow2004Subgroups}. The table can be completed by performing a case by case computation. 
    The notation $S_2\wr S_2$ represents the group $(S_2\times S_2)\rtimes S_2$.
\end{proof}

\begin{coro}
\label{cor:s5-most-neutral}
Every subgroup in Table~\ref{tab:s5-normalizers}, except $C_2^d$ and $D_{10}$, is neutral.
\end{coro}

\begin{proof}
For every other row, the normalizer extension splits. Apply Proposition~\ref{prop:normalizer} and Theorem~\ref{thm:coh-crit}.
\end{proof}

To finish, we need to study the two remaining cases. Let $G=C_2^d=\langle z\rangle$, where $z=(12)(34)$.  Choose
\[
    r=(1324),
    \quad
    s=(12).
\]
Then $r^2=z$ and
\[
    N_{S_5}(G)=\langle r,s\mid r^4=s^2=1,\ srs=r^{-1}\rangle\simeq D_8.
\]
Thus we have a short exact sequence
\begin{equation}
    1\ra \langle r^2\rangle\ra D_8\ra V_4\ra 1.
\end{equation}
We assert that $G$ is not neutral. To see this over an arbitrary base field $k$, take
$$
    F=K((x))((y)).
$$
Then $\Gal(F^{\mathrm{sep}}/F)\cong\hat{\bZ}\times\hat{\bZ}$. Consider
$$(r,s) \in \H^1(F,V_4) \cong \Hom(\hat{\bZ}\times \hat{\bZ},V_4).$$
We will prove that $(r,s)$ cannot be lifted to an element of $\H^1(F,D_8)$. If such a lift existed, it would necessarily have the form $(r^{2i+1}, r^{2j}s)$ for some $i,j \in \bZ$. However, a straightforward verification shows that the pair $(r^{2i+1}, r^{2j}s)$ is never commutative. So by \ref{thm:coh-crit}, $G$ is not neutral.

Let $G=D_{10}$, set
\[
    r=(12345),
    \quad
    a=(2354).
\]
Then $a$ has order $4$, $ara^{-1}=r^2$, and
\[
    F_{20}=\langle r,a\rangle,
    \qquad
    D_{10}=\langle r,a^2\rangle.
\]
Therefore the following sequnce is exact
\begin{equation}
\label{eq:f20-extension}
    1\ra D_{10}\ra F_{20}\ra C_2\ra 1.
\end{equation}

\begin{lemm}
\label{lem:f20-image}
For every field $F/k$,
\[
    \mathrm{im}(\H^1(F,F_{20})\ra \H^1(F,C_2))
    =
    \mathrm{im}(\H^1(F,C_4)\ra \H^1(F,C_2)).
\]
\end{lemm}

\begin{proof}
The projection $F_{20}=C_5\rtimes C_4\to C_4$ provides one embedding, and the subgroup $C_4=\langle a\rangle\subset F_{20}$ yields an embedding in the opposite direction.
\end{proof}

Consider now the short exact sequence
$$1\to C_2 \to C_4\to C_2\to 1.$$
When $\sqrt{-1}\in k$, this sequence can be identified with
$$1\to \mu_2\ra \mu_4\ra \mu_2\ra 1.$$
For any field extension $F/k$, the map $\H^1(F,\mu_4)\ra \H^1(F,\mu_2)$ is the homomorphism $$F^{\times}/F^{\times 4}\ra F^{\times}/F^{\times 2},$$ which is always surjective. Hence, by \ref{thm:coh-crit}, $D_{10}$ is neutral in this case.

Now assume $\sqrt{-1}\notin k$ and put $F=k((t))$. The quadratic extension $F(\sqrt t)/F$ does not embed in a cyclic quartic extension. Indeed, such a cyclic quartic extension would be totally ramified: otherwise its unique quadratic subextension would be unramified, whereas $F(\sqrt t)/F$ is ramified. It would also be tamely ramified. If $\sigma$ generated its Galois group and $\pi$ were a uniformizer, the residue of $\sigma(\pi)/\pi$ would be a primitive fourth root of unity in $k$, a contradiction. Thus the class of $F(\sqrt t)/F$ in $\H^1(F,C_2)$ is not in the image of $\H^1(F,C_4)$. By Lemma~\ref{lem:f20-image} and Theorem~\ref{thm:coh-crit}, the group $D_{10}$ is not neutral. Thus we obtain:

\begin{theo}
\label{thm:neutral-s5}
Up to conjugacy, every subgroup of $S_5$ is neutral except:
\begin{enumerate}[label=\textnormal{(\roman*)}]
    \item $\langle(12)(34)\rangle\simeq C_2$, which is never neutral;
    \item the transitive subgroup $D_{10}=\langle(12345),(25)(34)\rangle$, which is neutral precisely when $\sqrt{-1}\in k$.
\end{enumerate}
\end{theo}

\section{Proof of the main theorem}\label{sect:proof}
For simplicity, we now base change to $k = k_C$ and set $\cG=\cG_C$. Under this assumption, the curve $C$ is defined over its field of moduli precisely when $\cG$ has a $k$-rational point. As established in Section~\ref{sect:gerbes}, there always exists an algebraic group scheme $\Gamma$ over $k$ together with a morphism $\cG \to B\Gamma$ defined over $k$. In the hyperelliptic and trigonal cases we have $\Gamma = \PGL_2$; in the plane quintic case, $\Gamma = \PGL_3$; in the bielliptic case, $\Gamma = \GL_5 \times \bG_m$; and in the del Pezzo case, $\Gamma$ is the automorphism group of a split model of the corresponding surface. In this section we will exploit these morphisms in order to prove the main theorem stated in the introduction. 

We start by recalling two known cases, namely, the hyperelliptic and plane quintic cases.

\subsection{Hyperelliptic}

Let $\iota$ be the hyperelliptic involution. Huggins proves the following precise criterion.

\begin{theo}[Huggins]
\label{thm:huggins}
Let $C$ be hyperelliptic in characteristic different from $2$.  If
\[
    \Aut(C)/\langle\iota\rangle
\]
is non-cyclic, then $C$ is defined over its field of moduli.
\end{theo}

\begin{proof}[Reference]
This is \cite[Theorem~5.4]{Huggins2007}.  The same paper constructs examples with cyclic reduced automorphism group that do not descend.
\end{proof}

A gerbe-theoretic reformulation is given in \cite{BrescianiRational}.

\subsection{Plane quintics}

\begin{prop}[Bresciani]
\label{prop:plane-quintic}
Every smooth plane quintic is defined over its field of moduli.
\end{prop}

\begin{proof}
Every isomorphism between smooth plane curves of degree at least $4$ is induced by a projective linear transformation, so the plane embedding is intrinsic. We have used this fact to construct $\cG\to B\PGL_3$. By \cite[Theorem~1.1]{plane-curve}, every smooth plane curve whose degree is prime to $6$ is defined by a homogeneous polynomial over its field of moduli. Since $\gcd(5,6)=1$, the theorem applies.
\end{proof}

\subsection{Trigonal curves}

Let $f\colon C\ra \bP^1_K$ be the unique trigonal map (up to an isomorphism of $\bP^1_K$) and put $G=\Aut(C)$. Every automorphism preserves the pencil, so there is an exact sequence
\begin{equation}
\label{eq:trigonal-exact}
    1\ra I\ra G\ra H\ra 1,
\end{equation}
where $H\subset \PGL_2(K)$ and $I$ acts fibrewise.

\begin{lemm}
\label{lem:trigonal-kernel}
The group $I$ is either trivial or isomorphic to $C_3$.
\end{lemm}

\begin{proof}
On the open set where $f$ is unramified, $I$ acts faithfully on a fibre of three points, hence $I\subset S_3$. If $I$ contained an element of order $2$, the map $f$ would factor through a degree $2$ quotient of $C$, which is impossible because $\deg f=3$. Thus $I$ has odd order and is either trivial or $C_3$.
\end{proof}

\begin{prop}
\label{prop:trigonal-main}
Let $C$ be a smooth trigonal curve of genus $6$.  If $C$ is not defined over its field of moduli, then either $\Aut(C)$ is cyclic or there is an exact sequence
\[
    1\ra C_3\ra \Aut(C)\ra C_n\ra 1
\]
for some $n\geq 1$.
\end{prop}

\begin{proof}
We have shown that there exists a morphism $\cG\ra B\PGL_2$, geometrically, it corresponds to the trigonal morphism $f$.

Consider the canonical factorization $\cG \to \cH \to B\PGL_2$ of $\cG\ra B\PGL_2$, see \ref{canonical-factorization}. Let $\cP$ be the pullback of $[\bP^1/\PGL_2]$ along $\cH\to B\PGL_2$ and $\cX\ra \cG$ be the universal curve. Then we obtain a commutative diagram
\begin{center}
\begin{tikzcd}
\cX \arrow[r] \arrow[d] & \cP \arrow[d] \\
\cG \arrow[r]           & \cH          
\end{tikzcd}
\end{center}
where the upper horizontal morphism is nothing but the relative version of the trigonal morphism $f:C\ra \bP^1_K$.

Let $\bfX$ and $\bfP$ denote the coarse moduli spaces of $\cX$ and $\cP$, respectively. There is a natural morphism $\psi: \bfX \to \bfP$, which is a twisted form of the map $C/G \to \bP^1_K/H$. Since $f:C\ra \bP_K^1$ is a trigonal cover whose automorphism group is $I \cong C_3$, we obtain
\[
C/G \cong (C/I)/H \cong \bP^1_K/H,
\]
and therefore $\psi$ is an isomorphism as well.

Because $G$ acts faithfully on $C$, we obtain a rational section $\bfX \dashrightarrow \cX$. If $H$ is non-cyclic, we know that $\bfP \cong \bP^1_k$ by \cite[Theorem 5.2]{neutral-representation-dim-leq3}, and consequently $\bfX \cong \bP^1_k$. By the Lang–Nishimura theorem, this implies $\cX(k) \neq \emptyset$, and hence $\cG(k) \neq \emptyset$ as well.

Now assume that $I$ is trivial. In this situation, the action of $G$ on $\bP_K^1$ is generically faithful. Let $\cX'$ be the pullback of $[\bP^1/\PGL_2]$ along $\cG$. As before, we obtain a commutative diagram and a morphism of coarse moduli spaces $\psi' : \bfX' \to \bfP$, which is an isomorphism because it is a twisted form of the map $\bP^1_K/G \to \bP^1_K/H$. A similar argument show that $\cG(k)\neq \emptyset$ in this case as well.
\end{proof}

\begin{rema}
    The proof works for $p$-gonal genus $g$ curves as well, since all we need is $C/I\simeq \bP^1_K$ and that the pencil $p_g^1$ is unique.
\end{rema}

\subsection{Bielliptic} We refer to section~\ref{sect:bielliptic-gerbe-construction} for the construction of the faithful morphism $\cG\ra B\bG_m\times B\GL_5$. Consider the canonical factorization $\cG\ra \cH\ra \bG_m$ of the composite map $\cG\ra B\bG_m\times B\GL_5\ra B\bG_m$. By \cite[Proposition 5.1]{neutral-representation-dim-leq3}, the gerbe $\cH$ is neutral. Set $\cG_0:=\Spec k\times_{\cH}\cG$; it then suffices to show that $\cG_0$ is always neutral. 

Observe that the fundamental group of $\cG_0$ is precisely
\[
  J:=\ker\bigl(\chi_+:G\to\bG_m(K)\bigr),
\]
and the morphism $\cG_0\ra B\GL_5$ is a $j$-gerbe, where $j$ denotes the representation of $J$ on $\H^0(E,\omega_E\otimes L)$, see \ref{def:ggerbe}.

\begin{lemm}\label{lem:kernel-J}
The group $J$ in is cyclic and
\[
  J\cong C_2\quad\text{or}\quad C_{10}.
\]
\end{lemm}

\begin{proof}
Put
\[
  \overline G:=G/\langle\tau\rangle,
\]
where $\tau$ is the unique involution. This group acts faithfully on $E\simeq C/\langle \tau \rangle$. After choosing an origin on $E$, every element of $\overline G$ has a unique expression
\[
  t_a\circ\alpha,
  \quad a\in E(K),\quad \alpha\in\Aut(E,0).
\]
Translations act trivially on $\H^0(E,\omega_E)$ \cite[III.5, Proposition 5.1]{Silverman2009Arithmetic}. In characteristic $0$, the character
\[
  \Aut(E,0)\to \bG_m(K),
  \quad \alpha\mapsto \alpha^*|_{H^0(E,\omega_E)},
\]
is faithful \cite[Chapter III, \S5, Corollary 5.6]{Silverman2009Arithmetic}. Hence an element of $G$ lies in $J$ precisely when its image in $\overline G$ is a translation.

Let $T\subset\overline G$ be the subgroup consisting of translations. Since $\tau$ acts trivially on $\bE_+$, there is an exact sequence
\begin{equation}
\label{equation-J}
    1\to\langle\tau\rangle
  \to J
  \to T
  \to1.
\end{equation}
Every element of $T$ preserves the isomorphism class of $L$. Therefore
\[
  T\subset\ker\bigl(\varphi_L:E\to\Pic^0(E)\bigr),
  \quad
  \varphi_L(a)=t_a^*L\otimes L^{-1}.
\]
For a degree $5$ line bundle on an elliptic curve, $\ker(\varphi_L)=E[5]\cong C_5\times C_5$; see \cite[III.3, Proposition 3.4]{Silverman2009Arithmetic}. Thus $T\subset E[5]$.

A nontrivial translation acts freely on $E$. Since $T$ preserves the reduced branch divisor $B$ of degree $10$, the support of $B$ is a disjoint union of $T$-orbits. Hence $|T|$ divides $10$. As $T\subset E[5]$, it follows that
\[
  T=1\quad\text{or}\quad T\cong C_5.
\]
So $J\cong C_2$ or $C_{10}$.
\end{proof}

Set
\begin{equation}
    V := \H^0(E,\omega_E\otimes L),
    \quad \dim_K V = 5,
\end{equation}
equipped with the $J$-action obtained by restriction.

\begin{lemm}\label{lem:regular}
Suppose $J\simeq C_{10}$, and let $P\simeq C_5$ denote its unique subgroup of order $5$. Then the restricted representation $V|_P$ is the regular representation of $P$. In particular,
\[
  \dim_K V^P = 1.
\]
\end{lemm}

\begin{proof}
The group $P$ maps isomorphically to $T$ and therefore acts on $E$ by translations. Let
\[
  r:E\to F:=E/P
\]
be the quotient, so $r$ is a finite \'etale map of degree $5$. The action of $P\subset G$ on $C$ commutes with $\tau$. Hence it preserves the decomposition $\pi_*\mathcal O_C\cong \omega_E\oplus L^{-1}$, giving a $P$-action on $L$. By $fpqc$ descent along the $P$-torsor $r$, there is a line bundle $M$ on $F$ such that
\[
  r^*M\simeq L.
\]
Since $\deg r=5$ and $\deg L=5$, one has $\deg M=1$. Because $r$ is etale,
\[
  \omega_E\simeq r^*\omega_F,
\]
so
\[
  \omega_E\otimes L\simeq r^*(\omega_F\otimes M).
\]
Over the algebraically closed field $K$, the eigensheaf decomposition of the cyclic etale cover is
\[
  r_*\mathcal O_E\simeq
  \bigoplus_{\psi\in\widehat P}N_\psi^{-1},
\]
where each $N_\psi$ has degree $0$, and the summand indexed by $\psi$ is the $\psi$-eigensheaf, up to replacing all characters by their inverses. Using projection formula one has
\[
  V\simeq
  \bigoplus_{\psi\in\widehat P}
  \H^0\bigl(F,\omega_F\otimes M\otimes N_\psi^{-1}\bigr).
\]
Every line bundle $\omega_F\otimes M\otimes N_{\psi}^{-1}$ has degree $1$ on the elliptic curve $F$, and hence has a one-dimensional space of global sections. Thus every character of $P$ occurs exactly once in $V$. This is the regular representation, and the trivial character occurs once.
\end{proof}

\begin{prop}\label{prop:neutralness-of-J}
The representation
\[
  j:J\to\mathrm{GL}(V)
\]
is neutral.
\end{prop}

\begin{proof}
By Lemma~\ref{lem:kernel-J}, the group $J$ is cyclic. We now invoke \cite[Corollary 6.10]{BVY26}: let $C_p\subset J$ denote the unique subgroup of order $p$. Then the representation is neutral as long as 
\[
  \dim_K V - \dim_K V^{C_p} \not\equiv 0 \pmod p,
\]
for each prime $p$ dividing $|G|$.
For $p=2$, the subgroup $C_2$ is generated by $\tau$. Since $\tau$ acts as $-id$ on $V$, 
\[
  V^{C_2}=0,
  \quad
  \dim_KV-\dim_KV^{C_2}=5\not\equiv0\pmod2.
\]
If $J=C_2$, this proves the proposition.

Suppose $J=C_{10}$. For its order $5$ subgroup $P$, Lemma~\ref{lem:regular} gives
\[
  \dim_KV^P=1.
\]
Therefore
\[
  \dim_KV-\dim_KV^P=4\not\equiv0\pmod5.
\]
So the above condition holds for both prime divisors of $10$, the representation is neutral.
\end{proof}

\begin{coro}
    Let $C$ be a bielliptic curve over $K$, it descends to its field of moduli.
\end{coro}
\begin{proof}
    By \ref{prop:neutralness-of-J}, the $j$-gerbe $\cG_0\ra B\GL_5$ is neutral, hence $\cG=\cG_C$ is neutral.
\end{proof}

\subsection{Smooth del Pezzo surface} Next, suppose $C$ lies on a smooth del Pezzo surface of degree $5$. By Proposition~\ref{prop:del-pezzo-marking-morphism}, there is a faithful morphism $\cG\ra BS_5$ over $k$. Using the classification in Theorem~\ref{thm:neutral-s5}, we conclude
\begin{prop}
    Suppose the genus $6$ curve $C$ over $K$ lies on a smooth del Pezzo surface. Then it descends to its field of moduli unless either $\Aut(C)\cong C_2$, or $\Aut(C)\cong D_{10}$ and $\sqrt{-1}\notin k$.
\end{prop}

\subsection{Singular del Pezzo surface}

We first record the group-theoretic input needed in the singular case.

\subsubsection{Finite subgroups of split connected solvable groups}

We briefly recall the terminology. A linear algebraic group $\Gamma$ is \emph{solvable} if its derived series, obtained by repeatedly taking commutator subgroups, eventually becomes trivial. A linear algebraic group $U$ is \emph{unipotent} if, after base change to an algebraic closure and in a faithful linear representation, every element of $U$ is unipotent, or equivalently has all eigenvalues equal to $1$. A torus $T$ over $k$ is \emph{split} if $T\simeq\bG_m^r$ over $k$. A unipotent group $U$ over $k$ is \emph{split} if it has a filtration by normal $k$-subgroups whose successive quotients are isomorphic to $\bG_a$. In the situation below, saying that the connected solvable group $\Gamma$ is split means that its unipotent radical $U$ and a maximal torus $T$ are split over $k$, so that $\Gamma=U\rtimes T$ over $k$.

\begin{prop}
\label{prop:connected-split-solvable-neutral}
Let $\Gamma=U\rtimes T$ be a split connected solvable group over a field of characteristic $0$, where $U$ is split unipotent and $T$ is a split torus. Then every finite subgroup of $\Gamma(K)$ is neutral.
\end{prop}

\begin{proof}
Let $G\subset\Gamma(K)$ be finite. Its projection to $T(K)$ is injective, because a unipotent group over a field of characteristic $0$ has no nontrivial finite subgroups. Furthermore, $G$ is conjugate by some element of $U(K)$ to a subgroup of $T(K)$: indeed, let $H$ denote the image of $G$ in $T(K)$. Then both $H\subset U(K)\rtimes H$ and $G$ are Levi subgroups of $\Gamma(K)$, so one may invoke \cite[Proposition 5.4.1]{Conrad2014ReductiveGroupSchemes}.

Thus we may assume that $G\subset T(K)$, since by definition $G$ is neutral if and only if one of its $\Gamma(K)$-conjugates is neutral. Also, since $G$ is a finite subgroup of a torus, it must be Galois invariant, hence its normalizer is a gate; see \ref{prop:normalizer}. 

Because $T$ is abelian, an element of $\Gamma$ normalizes $G$ precisely when it centralizes $G$. Consequently,
$$
    N_\Gamma(G)=U^G\rtimes T.
$$
Because $G$ acts on $U$ by conjugation, the fixed-point subgroup $U^G = C_U(G)$ is a closed unipotent subgroup of $U$. In characteristic $0$, the exponential map gives an isomorphism, as varieties, between $U^G$ and the $G$-invariant subspace $\operatorname{Lie}(U)^G$. Hence $U^G$ is smooth and connected, and $U^G$ is split by \cite[Corollary 14.3.10]{SpringerLAG}.

Consequently,
$$
    N_\Gamma(G)/G = U^G \rtimes (T/G)
$$
is a split, connected, solvable group. By \cite[Theorem 2]{BrionPeyre}, any group of this type is special. Corollary~\ref{coro:special-normalizer} then shows that $G$ is neutral.
\end{proof}

\subsubsection{The five descent types}

Let $S_\tau$ be the split quintic del Pezzo surface of type $\tau$ chosen in Proposition~\ref{prop:singular-del-pezzo-marking-morphism}. The automorphism groups needed here are
$$
\begin{array}{c|c}
\tau & \Aut(S_\tau)\\
\hline
A_4 & U_3\rtimes\bG_m\\
A_3 & \bG_a^2\rtimes\bG_m\\
A_2+A_1 & B_2\times\bG_m\\
A_2 & B_2\rtimes C_2\\
A_1 & \bG_m\times D_6.
\end{array}
$$
where $U_3$ is the maximal unipotent subgroup of $\PGL_3$ and $B_2=\bG_a\rtimes\bG_m$ is the Borel subgroup of $\PGL_2$. Here $D_6\simeq S_3$ is denoted by $D_3$ in \cite{Virin2024}. These descriptions are given in \cite[Table~1, cases $36^\circ$--$39^\circ$ and $41^\circ$]{Virin2024}.

\begin{theo}
\label{thm:neutral-singular-del-pezzo}
If $\tau$ is $A_4$, $A_3$, $A_2+A_1$, $A_2$, or $A_1$, then every finite subgroup of $\Aut(S_\tau)(K)$ is neutral.
\end{theo}

\begin{proof}
For the types $A_4$, $A_3$, and $A_2+A_1$, the displayed descriptions show that $\Aut(S_\tau)$ is split, connected, and solvable. Hence Proposition~\ref{prop:connected-split-solvable-neutral} applies.

Suppose that $\tau=A_2$. Let $B_2$ denote the standard Borel subgroup of $\PGL_2$, written as
$$
B_2=\bG_a\rtimes\bG_m,
\quad
t\cdot a=ta.
$$
By \cite[Claim~4.11]{Virin2024}, the nontrivial element $\sigma$ of $C_2$ acts on $B_2$ by
$$
\sigma(a,t)=(-a,t).
$$
Consequently, we may write
$$
\Gamma:=\Aut(S_\tau)=V\rtimes D,
\quad
V=\bG_a,
\quad
D=\bG_m\times C_2,
$$
where
$$
(t,\sigma^\epsilon)\cdot a=(-1)^\epsilon ta.
$$

Let $G\subset\Gamma(K)$ be finite. As in the proof of \ref{prop:connected-split-solvable-neutral} we may assume that $G\subset D(K)$. Because $D$ is abelian, an element of $\Gamma$ normalizes $G$ precisely when its $V$-component is fixed by $G$. Hence
$$
N_\Gamma(G)=V^G\rtimes D.
$$
Moreover, $G$ acts trivially on $V^G$, so
$$
N_\Gamma(G)/G=V^G\rtimes(D/G).
$$
For every field extension $F/k$, projection onto the second factor induces bijections
$$
\H^1(F,V^G\rtimes D)\xrightarrow{\sim}\H^1(F,D)
$$
and
$$
\H^1(F,V^G\rtimes(D/G))
\xrightarrow{\sim}
\H^1(F,D/G).
$$
Indeed, the fiber over a given class $\alpha$, meaning the collection of $1$-cocycles lying above $\alpha$, is parametrized by the first cohomology group of the corresponding twist ${}_\alpha V^G$. This twist is either trivial or a vector group, and in either case its first Galois cohomology vanishes. Via these identifications, the map
$$
\H^1(F,N_\Gamma(G))
\to
\H^1(F,N_\Gamma(G)/G)
$$
corresponds to
$$
\H^1(F,D)\to\H^1(F,D/G).
$$

Now consider the projection $\epsilon:G\to C_2$. If $\epsilon$ is the trivial homomorphism, then $G=\mu_n\subset\bG_m$ for some $n$, and we have
$$
D/G\simeq(\bG_m/\mu_n)\times C_2
\simeq\bG_m\times C_2.
$$
Since $\H^1(F,\bG_m)=1$, the map
$$
\H^1(F,D)\to\H^1(F,D/G)
$$
reduces to the identity on $\H^1(F,C_2)$, and hence is surjective.

If instead $\epsilon$ is surjective, then $D/G\cong\bG_m$, so $\H^1(F,D/G)=1$, and once again the induced map on cohomology is surjective. By Proposition~\ref{prop:normalizer} and Theorem~\ref{thm:coh-crit}, we conclude that $G$ is neutral.

It remains to treat $\tau=A_1$. Write
$$
\Gamma=\Aut(S_\tau)=T\times D_6,
\quad
T=\bG_m,
$$
and let $G\subset\Gamma(K)$ be finite. Set $H=\mathrm{im}(G\ra D_6)$ and $A=\ker(G\ra D_6)$. The image of $G$ in $(T/A)\times H$ intersects $T/A$ trivially and projects isomorphically onto $H$. It is therefore the graph of a character
$$
\chi:H\to T/A.
$$
Consequently, $G$ is the inverse image of this graph under
$$
T\times H\to(T/A)\times H.
$$
Put
$$
Q=
\left\{
q\in N_{D_6}(H):
\chi(qhq^{-1})=\chi(h)
\text{ for every }h\in H
\right\}.
$$
Then $N_\Gamma(G)=T\times Q$. Up to conjugacy, the subgroups of $D_6\cong S_3$ are
$$
1,\quad C_2,\quad C_3,\quad D_6.
$$
If $Q=H$, the homomorphism
$$
T\times H\to T/A,
\quad
(t,h)\mapsto(t\bmod A)\chi(h)^{-1},
$$
has kernel $G$ and therefore identifies
$$
N_\Gamma(G)/G\simeq T/A.
$$
Since $T/A$ is a split torus, its first Galois cohomology is trivial.

There are only two cases in which $Q\neq H$. If $H=1$, then
$$
N_\Gamma(G)/G\simeq(T/A)\times D_6,
$$
and the map on first Galois cohomology is the identity on the $D_6$-factor. If $H=C_3$ and $\chi$ is trivial, then
$$
N_\Gamma(G)/G\simeq(T/A)\times C_2,
$$
and the map on finite factors is induced by the split quotient
$$
D_6\to C_2.
$$
It follows that
$$
\H^1(F,N_\Gamma(G))
\to
\H^1(F,N_\Gamma(G)/G)
$$
is surjective for every field extension $F/k$ in all cases. Proposition~\ref{prop:normalizer} and Theorem~\ref{thm:coh-crit} show that $G$ is neutral.
\end{proof}

Finally, assume that $C$ is non-special and lies on a singular quintic del Pezzo surface $S_C$, whose singularity type is $\tau$. Proposition~\ref{prop:singular-del-pezzo-marking-morphism} gives a faithful morphism
$$
    \cG\ra B\Aut(S_\tau),
$$
whose geometric inertia is the embedding
$$
    \Aut(C)=\Aut(S_C,C)\hookrightarrow\Aut(S_\tau)(K).
$$

\begin{prop}
\label{prop:singular-del-pezzo-descent}
If $\tau$ is $A_4$, $A_3$, $A_2+A_1$, $A_2$, or $A_1$, then $C$ descends to its field of moduli.
\end{prop}

\begin{proof}
By Theorem~\ref{thm:neutral-singular-del-pezzo}, every finite subgroup of $\Aut(S_\tau)(K)$ is neutral for the five displayed types. So $\cG(k)\neq \emptyset$.
\end{proof}

The last remaining singularity type to consider is $2A_1$. In this case, the automorphism group of the surface is $\bG_m^2\rtimes C_2$, where the nontrivial element of $C_2$ swaps the two tori, and the associated normalizer quotient $N_{\Gamma}(G)/G$ may fail to be special. For example, take $G=\langle (i,-i,id)\rangle\subset \Gamma(K)$. The normalizer of $G$ is the entire group $\Gamma(K)$. On the other hand, $G$ is isomorphic to the subgroup
$$G'=\langle \begin{pmatrix}
    i&0\\
    0&-i
\end{pmatrix}\rangle\subset \GL_2(K),$$
whose normalizer in $\GL_2(K)$ is exactly $\bG_m^2\rtimes C_2$. We have shown that $G'$ is not a neutral subgroup of $\GL_2(K)$; see \cite[Theorem 5.3]{neutral-representation-dim-leq3}.

\printbibliography
\end{document}